\documentclass{article}

\usepackage{amsmath,amsthm}
\usepackage{amsfonts}
\usepackage{amssymb}

\theoremstyle{plain}
\newtheorem{theorem}{Theorem}

\newtheorem{lemma}{Lemma}
\theoremstyle{remark}

\newtheorem{remark}{Remark}
\theoremstyle{definition}
\newtheorem{definition}{Definition}
\newtheorem{example}{Example}

\begin{document}

\title{Existence of positive solutions to a discrete fractional boundary
value problem and corresponding Lyapunov-type inequalities\thanks{This 
is a preprint of a paper whose final and definite form is with 
journal \emph{Opuscula Mathematica}, vol.~38, no.~1 (2018),
ISSN 1232-9274, e-ISSN 2300-6919, 
available at {\tt http://dx.doi.org/10.7494/OpMath}.\newline
Submitted 15-Nov-2016; Revised 31-May-2017; Accepted 18-June-2017.}}

\author{Amar Chidouh$^1$\\
\texttt{m2ma.chidouh@gmail.com} 
\and Delfim F. M. Torres$^2$\thanks{Corresponding author.} \\
\texttt{delfim@ua.pt}}

\date{$^1$Laboratory of Dynamic Systems,\\
Houari Boumedienne University,\\
Algiers, Algeria\\
[0.3cm] $^2$Center for Research and Development in Mathematics and
Applications (CIDMA), Department of Mathematics,\\
University of Aveiro, 3810-193 Aveiro, Portugal}

\maketitle


\begin{abstract}
We prove existence of positive solutions to a boundary value problem
depending on discrete fractional operators. Then, corresponding discrete
fractional Lyapunov-type inequalities are obtained.

\bigskip

\noindent \textbf{Keywords:} fractional difference equations, Lyapunov-type
inequalities, fractional boundary value problems, positive solutions.

\medskip

\noindent \textbf{MSC 2010:} 26A33, 26D15, 39A12.
\end{abstract}


\section{Introduction}

Recently, a large debate appeared regarding Lyapunov-type inequalities --
see, e.g., \cite{MyID:345,fdf,MR3371249,MR3352688} and references therein.
In 1907, Lyapunov proved in \cite{MR1508297} that if $q:[a,b]\rightarrow
\mathbb{R}$ is a continuous function, then a necessary condition for the
boundary value problem
\begin{equation}  
\label{11}
\left\{
\begin{array}{c}
y^{\prime\prime}+qy=0,\quad a<t<b, \\
y(a)=y(b)=0
\end{array}
\right.
\end{equation}
to have a nontrivial solution is given by
\begin{equation*}
\int\limits_{a}^{b}\left\vert q(s)\right\vert ds>\frac{4}{b-a}.
\end{equation*}
Ferreira has succeed to generalize the above classical result to the case
when the second-order derivative in \eqref{11} is substituted by a
fractional operator of order $\alpha$, in Caputo or Riemann--Liouville sense
\cite{MR3124347,ferreira2014lyapunov}. More recently, the authors obtained
in \cite{MyID:345} a generalized Lyapunov-type inequality for the following
fractional boundary value problem:
\begin{equation}  
\label{p}
\left\{
\begin{array}{c}
_{a}D^{\alpha}y+q(t)f(y)=0,\quad a<t<b, \\
y(a)=y(b)=0,
\end{array}
\right.
\end{equation}
where $_{a}D^{\alpha}$ is the Riemann--Liouville derivative, $1<\alpha\leq2$, 
and $q:[a,b]\rightarrow\mathbb{R}_{+}$ is a Lebesgue integrable function.

\begin{theorem}[See \protect\cite{MyID:345}]
\label{th0 copy(1)} 
Let $q:[a,b]\rightarrow\mathbb{R}_{+}$ be a real
Lebesgue integrable function. Assume that $f\in C\left( \mathbb{R}_{+},
\mathbb{R}_{+}\right) $ is a concave and nondecreasing function. If the
fractional boundary value problem \eqref{p} has a nontrivial solution, then
\begin{equation}
\label{56}
\int\limits_{a}^{b}q(t)dt 
> \frac{4^{\alpha-1}
\Gamma(\alpha)\eta}{(b-a)^{\alpha -1}f(\eta)}, 
\end{equation}
where $\eta=\max_{t\in\lbrack a,b]}y(t)$.
\end{theorem}

Here we are concerned with the discrete fractional calculus 
\cite{MR2728463,MR3445243}. It turns out that Lyapunov fractional inequalities
can also be obtained by considering a discrete fractional difference in 
\eqref{11} instead the Caputo or Riemann--Liouville derivatives \cite{fdf}.
Motivated by the results obtained in \cite{fff,MyID:345,fdf,ff}, we prove
here some generalizations of the Lypunov inequality of \cite{fdf}. The new
inequalities are, in some sense, similar to that of \eqref{56} (compare with
\eqref{eq:new:ineq} and \eqref{eq:LI:d2}) but, instead of \eqref{p}, they
involve the following discrete fractional boundary value problem:
\begin{equation}  
\label{cd}
\left\{
\begin{array}{c}
\triangle^{\alpha}y+q(t+\alpha-1)f(y(t+\alpha-1))=0, \quad 1<\alpha\leq2, \\
y(\alpha-2)=y(\alpha+b+1)=0, \quad b\geq2, \quad b\in \mathbb{N},
\end{array}
\right.
\end{equation}
where operator $\triangle^{\alpha}$ is defined in Section~\ref{sec:2}.
Interestingly, we show that the hypothesis found in Theorem~\ref{th0 copy(1)}, 
assuming the nonlinear term $f$ to be concave, can be removed in the
discrete setting (see Theorems~\ref{th0} and \ref{co}).

The paper is organized as follows. In Section~\ref{sec:2}, we recall some
notations, definitions and preliminary facts, which are used throughout the
work. Our original results are then given in Section~\ref{sec:3}: using the
Guo--Krasnoselskii fixed point theorem, we establish in Section~\ref{sec:3.1}
an existence result for the discrete fractional boundary value problem 
\eqref{cd} (see Theorem~\ref{th}); then, in Section~\ref{sec:3.2}, assuming
that function $f:\mathbb{R}_{+}\rightarrow \mathbb{R}_{+}$ is only
continuous and nondecreasing, we generalize the Lyapunov inequality given in
\cite[Theorem~3.1]{fdf} (see Theorems~\ref{th0} and \ref{co}). Examples
illustrating the new results are given.


\section{Preliminaries}
\label{sec:2}

In this section, we recall some notations, definitions and preliminary
facts, which are used throughout the text. We begin by recalling the
well-known definition of power function:
\begin{equation*}
x^{[y]}=\frac{\Gamma(x+1)}{\Gamma (x-y+1)}
\end{equation*}
for any $x$ and $y$ for which the right-hand side is defined. We borrow from
\cite{MyID:179} the following notation:
\begin{equation*}
\mathbb{N}_{a}:=\{a,a+1,a+2,...\}, \ a\in \mathbb{R}.
\end{equation*}

\begin{definition}
For a function $f:\mathbb{N}_{a}\rightarrow \mathbb{R}$, the discrete
fractional sum of order $\alpha\geq0$ is defined by
\begin{equation*}
(_{a}\triangle_{t}^{-\alpha}f)(t)=\frac{1}{\Gamma(\alpha)}
\sum\limits_{s=a}^{t-\alpha}(t-s-1)^{[\alpha-1]}f(s), 
\quad t\in\mathbb{N}_{a+\alpha}.
\end{equation*}
\end{definition}

\begin{definition}
For a function $f: \mathbb{N}_{a}\rightarrow \mathbb{R}$, the discrete
fractional difference of order $\alpha>0$ $(n-1\leq\alpha\leq n$, where 
$n\in \mathbb{N})$ is defined by
\begin{equation*}
(\triangle^{\alpha}f)(t)=(\triangle^{n} \,
_{a}\triangle_{t}^{-(n-\alpha)}f)(t), \quad t\in \mathbb{N}_{a+n-\alpha},
\end{equation*}
where $\triangle^{n}$ is the standard forward difference of order $n$.
\end{definition}

The reader interested on more details about the discrete fractional calculus
is referred to \cite{fff,fdf,ff,MR3445243}.

\begin{definition}
Let $X$ be a real Banach space. A nonempty closed convex set $P\subset X$ is
called a cone if it satisfies the following two conditions:

\begin{description}
\item[$(i)$] $x\in P$, $\lambda\geq0$, implies $\lambda x\in P$;

\item[$(ii)$] $x\in P$, $-x\in P$, implies $x=0$.
\end{description}
\end{definition}

\begin{lemma}[{Guo--Krasnoselskii fixed point theorem \protect\cite{[3]}}]
\label{kras} 
Let $X$ be a Banach space and let $K\subset X$ be a cone.
Assume $\Omega_{1}$ and $\Omega_{2}$ are bounded open subsets of $X$ with 
$0\in\Omega_{1}\subset\overline{\Omega}_{1}\subset$ $\Omega_{2}$, and let 
$T:K\cap (\overline{\Omega}_{2}\backslash\Omega_{1})\rightarrow K$ be a
completely continuous operator such that
\begin{description}
\item[$(i)$] $\left\Vert Tu\right\Vert \geq\left\Vert u\right\Vert $ for any
$u\in K\cap\partial\Omega_{1}$ and $\left\Vert Tu\right\Vert \leq\left\Vert
u\right\Vert $ for any $u\in K\cap\partial\Omega_{2}$; or

\item[$(ii)$] $\left\Vert Tu\right\Vert \leq\left\Vert u\right\Vert $ for
any $u\in K\cap\partial\Omega_{1}$ and $\left\Vert Tu\right\Vert
\geq\left\Vert u\right\Vert $ for any $u\in K\cap\partial\Omega_{2}$.
\end{description}
Then $T$ has a fixed point in $K\cap(\overline{\Omega}_{2}
\backslash \Omega_{1})$.
\end{lemma}


\section{Main results}
\label{sec:3}

Let us consider the nonlinear discrete fractional boundary value problem
\eqref{cd}. We deal with its sum representation involving a Green function.

\begin{lemma}
\label{lemmaRF} 
Function $y$ is a solution to the boundary value problem
\eqref{cd} if, and only if, $y$ satisfies
\begin{equation*}
y(t)=\sum \limits_{s=0}^{b+1}G(t,s)q(s+\alpha-1)f(y(s+\alpha-1)),
\end{equation*}
where
\begin{equation}  
\label{eq:Gf}
G(t,s)=\frac{1}{\Gamma(\alpha)}
\begin{cases}
\frac{t^{[\alpha-1]}(\alpha+b-s)^{[\alpha-1]}}{(\alpha+b+1)^{\left[ 
\alpha-1\right] }}-(t-s-1)^{[\alpha-1]},\ s<t-\alpha+1\leq b+1, \\
\frac{t^{[\alpha-1]}(\alpha+b-s)^{[\alpha-1]}}{(\alpha+b+1)^{\left[ 
\alpha-1 \right] }},\quad t-\alpha+1<s\leq b+1,
\end{cases}
\end{equation}
is the Green function associated to problem \eqref{cd}.
\end{lemma}

\begin{proof}
Similar to the one found in \cite{fff}.
\end{proof}

\begin{lemma}
\label{lmm} 
The Green function $G$ given by \eqref{eq:Gf} 
satisfies the following properties:
\begin{enumerate}
\item $G(t,s)>0$ for all $t\in\lbrack\alpha-1, 
\alpha+b]_{\mathbb{N}_{\alpha-1}}$ and $s\in\lbrack1,b+1]_{\mathbb{N}_{1}}$;

\item $\max_{t\in\lbrack\alpha-1, \alpha+b]_{\mathbb{N}_{\alpha-1}}}G(t,s)
=G(s+\alpha-1,s),\ s\in\lbrack1,b+1]_{\mathbb{N}_{1}}$;

\item $G(s+\alpha-1,s)$ has a unique maximum given by
\begin{equation*}
\max_{s\in\lbrack1,b+1]_{\mathbb{N}_{1}}}G(s+\alpha-1,s) =\left\{
\begin{array}{l}
\frac{1}{4}\frac{(b+2\alpha)(b+2)\Gamma^{2}(\frac{b}{2}+\alpha)\Gamma (b+3)}{
\Gamma(\alpha)\Gamma(b+\alpha+2)\Gamma^{2}(\frac{b}{2}+2)},\ \text{if } b
\text{ is even,} \\
\frac{1}{\Gamma(\alpha)}\frac{\Gamma(b+3)\Gamma^{2}(\frac{b+1}{2}+\alpha )}{
\Gamma(b+\alpha+2)\Gamma^{2}(\frac{b+3}{2})}, \ \text{if }b\text{ is odd};
\end{array}
\right.
\end{equation*}

\item there exists a positive constant $\lambda\in(0,1)$ such that
\begin{equation*}
\min_{t\in\left[ \frac{b+\alpha}{4},\frac{3(b +\alpha)}{4}
\right]_{_{\mathbb{N}_{\alpha-1}}}}G(t,s) \geq\lambda
\max_{t\in\lbrack\alpha-1, \alpha+b]_{\mathbb{N}_{\alpha-1}}}
G(t,s) =\lambda G(s+\alpha-1,s)
\end{equation*}
for $s\in\lbrack1,b+1]_{\mathbb{N}_{1}}$.
\end{enumerate}
\end{lemma}

\begin{proof}
Similar to the one found in \cite{fdf}.
\end{proof}


\subsection{Existence of positive solutions}
\label{sec:3.1}

Let us consider the Banach space
\begin{equation*}
X:=\left\{ y:[\alpha-2,\alpha+b+1]_{\mathbb{N}_{\alpha-2}} 
\rightarrow \mathbb{R},\ y(\alpha-2)=y(\alpha+b+1)=0\right\}
\end{equation*}
with the supremum norm. In agreement with Lemma~\ref{kras}, to prove
existence of a solution to the discrete fractional boundary value problem
\eqref{cd}, it suffices to prove that a suitable map $T$ has a fixed point
in $X$. We are interested to prove existence of nontrivial positive
solutions to \eqref{cd}, which are the ones to have a physical meaning 
\cite{MR3480533}. For that, we consider the following two hypotheses:
\begin{description}
\item[$(H_1)$] $f(y)\geq\overset{\ast}{\gamma}r_{1}$ for $y\in\lbrack0,r_{1}]$,

\item[$(H_2)$] $f(y)\leq\gamma r_{2}$ for $y\in\lbrack0,r_{2}]$,
\end{description}
where $f:\mathbb{R}_{+}\rightarrow \mathbb{R}_{+}$ is continuous. 
In what follows, we take
\begin{equation}  
\label{eq:gamma}
\gamma :=\left( \sum\limits_{s=0}^{b+1}
G(s+\alpha -1,s)q(s+\alpha -1)\right)^{-1}
\end{equation}
and
\begin{equation}  
\label{eq:ast:gamma}
\overset{\ast }{\gamma } :=\left( 
\sum\limits_{s=\frac{b+\alpha }{4}}^{\frac{3(b+\alpha )}{4}} 
\lambda G(s+\alpha -1,s)q(s+\alpha -1)\right)^{-1}.
\end{equation}

\begin{theorem}
\label{th} 
Let $q:$ $\mathbb{[\alpha -}1, \alpha +b]_{\mathbb{N}_{\alpha-1}}
\rightarrow \mathbb{R}_{+}$ be a nontrivial function. Assume that there
exist two positive constants $r_{2}>r_{1}>0$ such that the assumptions $(H_1)$
and $(H_2)$ are satisfied. Then the discrete fractional boundary value
problem \eqref{cd} has at least one nontrivial positive solution $y$
belonging to $X$ such that $r_{1}\leq \left\Vert y\right\Vert \leq r_{2}$.
\end{theorem}

\begin{proof}
First of all, we define the operator $T:X\rightarrow X$ as follows:
\begin{equation}
\label{T}
Ty(t)=\sum\limits_{s=0}^{b+1}G(t,s)q(s+\alpha -1)f(y(s+\alpha -1)).
\end{equation}
We use Lemma~\ref{kras} with the following cone $K$:
\begin{equation*}
K:=\left\{ y\in X:\ \min_{\left[ \frac{b+\alpha }{4},
\frac{3(b+\alpha )}{4}\right]_{_{\mathbb{N}_{\alpha -1}}}}y(t)
\geq \lambda \left\Vert y\right\Vert \right\}.
\end{equation*}
To prove existence of a nontrivial solution to the fractional discrete problem
\eqref{cd} amounts to show existence of a fixed point to the operator $T$ in
$K\cap(\overline{\Omega }_{2}\backslash \Omega _{1})$.
From Lemma~\ref{lmm}, we get that $T(K)\subset K$. Taking into account that $T$ is a
summation operator on a discrete finite set, it follows that $T:K\rightarrow K$
is a completely continuous operator. Now, it remains to consider the
first part $(i)$ of Lemma~\ref{kras} to prove our result. Let
$\Omega _{i}=\left\{ y\in K:\left\Vert y\right\Vert \leq r_{i}\right\}$.
From $(H_1)$, we have for $t\in \left[ \frac{b+\alpha }{4},
\frac{3(b+\alpha )}{4}\right]_{_{\mathbb{N}_{\alpha -1}}}$
and $y\in K\cap \partial \Omega _{1}$ that
\begin{align*}
(Ty)(t)& \geq \sum\limits_{s=0}^{b+1}\min_{t\in \left[ \frac{b+\alpha }{4},
\frac{3(b+\alpha )}{4}\right]_{_{\mathbb{N}_{\alpha -1}}}}
G(t,s)q(s+\alpha -1)f(y(s+\alpha -1)) \\
& \geq \overset{\ast }{\gamma }\left( \sum\limits_{s=0}^{b+1}\lambda
G(s+\alpha -1,s)q(s+\alpha -1)\right) r_{1}\\
&=\left\Vert y\right\Vert.
\end{align*}
Thus, $\left\Vert Ty\right\Vert \geq \left\Vert y\right\Vert $ for $y\in
K\cap \partial \Omega _{1}$.
Let us now prove that $\left\Vert Ty\right\Vert \leq \left\Vert y\right\Vert
$ for all $y\in K\cap \partial \Omega _{2}$. From $(H_2)$, it follows that
\begin{align*}
\left\Vert Ty\right\Vert 
&= \max_{t\in \lbrack \alpha -1,\alpha +b]_{\mathbb{N}_{\alpha -1}}}
\sum\limits_{s=0}^{b+1}G(t,s)q(s+\alpha -1)f(y(s+\alpha -1))\\
& \leq \gamma \left( \sum\limits_{s=0}^{b+1}G(s+\alpha -1,s)q(s+\alpha-1)\right) r_{2}\\
&=\left\Vert y\right\Vert
\end{align*}
for $y\in K\cap \partial \Omega _{2}$. Thus, from Lemma~\ref{kras}, we
conclude that the operator $T$ defined by \eqref{T} has a fixed point in
$K\cap (\overline{\Omega }_{2}\backslash \Omega _{1})$. Therefore, the
discrete fractional boundary problem \eqref{cd} has at least one positive
solution belonging to $X$ such that $r_{1}\leq \left\Vert y\right\Vert \leq r_{2}$.
\end{proof}

\begin{example}
Consider the following discrete fractional boundary value problem:
\begin{equation}
\label{pex}
\left\{
\begin{array}{c}
\triangle^{\frac{3}{2}}y(t)+\frac{(\frac{2t+1}{2})}{y(\frac{2t+1}{2})+20}=0,
\\
y(-\frac{1}{2})=y(\frac{11}{2})=0.
\end{array}
\right.  
\end{equation}
Note that this problem is of type \eqref{cd} with $b=3$. 
The value \eqref{eq:gamma} of $\gamma$ is given by
\begin{equation*}
\gamma=\left( \sum \limits_{s=0}^{4}\frac{1}{\Gamma(\frac{3}{2})} 
\frac{\Gamma(6)\Gamma^{2}\left(2+\frac{3}{2}\right)}{\Gamma(5+\frac{3}{2})
\Gamma^{2}(3)}\left(\frac{2s+1}{2}\right)\right) ^{-1} \approx 0.0616
\end{equation*}
while, from formula (3.3) of \cite{fff}, the value \eqref{eq:ast:gamma} 
of $\overset{\ast}{\gamma}$ becomes
\begin{equation*}
\overset{\ast}{\gamma}=\left( \sum \limits_{s=0}^{4}0.03779 
\frac{1}{\Gamma(\frac{3}{2})}\frac{\Gamma(6)\Gamma^{2}(2
+\frac{3}{2})}{\Gamma(5 +\frac{3}{2})\Gamma^{2}(3)}
\left(\frac{2s+1}{2}\right)\right) ^{-1} \approx 1.6301.
\end{equation*}
Choose $r_{1}=1/100$ and $r_{2}=1$. Then, one gets

\begin{enumerate}
\item $f(y)\geq\overset{\ast}{\gamma}r_{1}$ for $y\in\lbrack0,1/100]$;

\item $f(y)\leq\gamma r_{2}$ for $y\in\lbrack0,1]$.
\end{enumerate}
Therefore, from Theorem~\ref{th}, problem \eqref{pex} has at least one
nontrivial solution $y$ in $X$ such that $y\in\lbrack1/100,1]$.
\end{example}


\subsection{Generalized discrete fractional Lyapunov inequalities}
\label{sec:3.2}

The next result generalizes \cite[Theorem 3.1]{fdf}: in the particular case
of $f(y)=y$, inequalities \eqref{eq:new:ineq} reduce to those in 
\cite[Theorem~3.1]{fdf}. Note that $f\in C(\mathbb{R}_{+},\mathbb{R}_{+})$ 
is a nondecreasing function.

\begin{theorem}
\label{th0} 
Let $q:$ $\mathbb{[\alpha -}1,\alpha +b]_{\mathbb{N}_{\alpha-1}}
\rightarrow \mathbb{R}$ be a nontrivial function. Assume that 
$f\in C\left( \mathbb{R}_{+},\mathbb{R}_{+}\right)$ 
is a nondecreasing function. If the discrete fractional boundary 
value problem \eqref{cd} has a nontrivial solution $y$, then
\begin{equation}
\label{eq:new:ineq}
\left\{
\begin{array}{l}
\sum\limits_{s=0}^{b+1}\left\vert q(s+\alpha -1)\right\vert >\frac{4\Gamma
(\alpha )\Gamma (b+\alpha +2)\Gamma ^{2}(\frac{b}{2}+2)\eta }{(b+2\alpha
)(b+2)\Gamma ^{2}(\frac{b}{2}+\alpha )\Gamma (b+3)f\left( \eta \right) },
\text{ if }b\text{ is even}, \\[0.3cm]
\sum\limits_{s=0}^{b+1}\left\vert q(s+\alpha -1)\right\vert >\frac{\Gamma
(\alpha )\Gamma (b+\alpha +2)\Gamma ^{2}(\frac{b+3}{2})\eta }{\Gamma
(b+3)\Gamma ^{2}(\frac{b+1}{2}+\alpha )f\left( \eta \right) },\text{ if }b
\text{ is odd},
\end{array}
\right.   
\end{equation}
where $\eta =\max_{[\alpha -1,\alpha +b]_{\mathbb{N}_{\alpha -1}}}
y(s+\alpha-1)$.
\end{theorem}

\begin{proof}
Since the discrete fractional problem \eqref{cd} has a nontrivial solution,
we get via Lemma~\ref{lemmaRF} that
\begin{align*}
\left\Vert y\right\Vert & \leq \sum\limits_{s=0}^{b+1}G(s+\alpha
-1,s)\left\vert q(s+\alpha -1)\right\vert f(y(s+\alpha -1))\\
& \leq \left\{
\begin{array}{l}
\frac{(b+2\alpha )(b+2)\Gamma ^{2}(\frac{b}{2}+\alpha )\Gamma
(b+3)}{4 \Gamma (\alpha )\Gamma (b+\alpha +2)\Gamma^{2}(\frac{b}{2}+2)}
\sum\limits_{s=0}^{b+1}\left\vert q(s+\alpha -1)\right\vert f(y(s+\alpha
-1))\text{ if }b\text{ is even}, \\
\frac{1}{\Gamma (\alpha )}\frac{\Gamma (b+3)\Gamma^{2}(\frac{b+1}{2}+\alpha
)}{\Gamma (b+\alpha +2)\Gamma ^{2}(\frac{b+3}{2})}\sum\limits_{s=0}^{b+1}\left\vert
q(s+\alpha -1)\right\vert f(y(s+\alpha -1))\text{ if }b\text{ is odd}.
\end{array}
\right.
\end{align*}
Taking into account that $f$ is a nondecreasing function and
$$
\eta =\max_{[\alpha -1,\alpha +b]_{\mathbb{N}_{\alpha -1}}}y(s+\alpha -1),
$$
we get that
\begin{equation*}
\left\Vert y\right\Vert <\left\{
\begin{array}{l}
\frac{1}{4}\frac{(b+2\alpha )(b+2)\Gamma ^{2}(\frac{b}{2}+\alpha )\Gamma
(b+3)}{\Gamma (\alpha )\Gamma (b+\alpha +2)\Gamma ^{2}(\frac{b}{2}+2)}
\sum\limits_{s=0}^{b+1}\left\vert q(s+\alpha -1)\right\vert f\left( \eta
\right) \text{ if }b\text{ is even}, \\
\frac{1}{\Gamma (\alpha )}\frac{\Gamma (b+3)\Gamma ^{2}(\frac{b+1}{2}+\alpha
)}{\Gamma (b+\alpha +2)\Gamma ^{2}(\frac{b+3}{2})}
\sum\limits_{s=0}^{b+1}\left\vert q(s+\alpha -1)\right\vert
f\left( \eta \right) \text{ if }b\text{ is odd}.
\end{array}
\right.
\end{equation*}
Hence,
\begin{equation*}
\left\{
\begin{array}{l}
\sum\limits_{s=0}^{b+1}\left\vert q(s+\alpha -1)\right\vert >\frac{4\Gamma
(\alpha )\Gamma (b+\alpha +2)\Gamma ^{2}(\frac{b}{2}+2)\eta }{(b+2\alpha
)(b+2)\Gamma ^{2}(\frac{b}{2}+\alpha )\Gamma (b+3)f\left( \eta \right)}
\text{ if }b\text{ is even}, \\
\sum\limits_{s=0}^{b+1}\left\vert q(s+\alpha -1)\right\vert >\frac{\Gamma
(\alpha )\Gamma (b+\alpha +2)\Gamma ^{2}(\frac{b+3}{2})\eta }{\Gamma
(b+3)\Gamma ^{2}(\frac{b+1}{2}+\alpha )f\left( \eta \right) }
\text{ if }b\text{ is odd}.
\end{array}
\right.
\end{equation*}
This concludes the proof.
\end{proof}

\begin{remark}
Lyapunov inequalities are usually used to get bounds for
the eigenvalues of Sturm--Liouville problems \cite{MR3443424,MyID:376}. 
Therefore, if we consider the discrete Sturm--Liouville problem \eqref{cd} 
with $f(y)=y$ and $q(t)=\lambda$, then inequalities \eqref{eq:new:ineq} 
give us an interval for the eigenvalues $\lambda$ \cite{fdf}. Here we do 
a generalization of the results obtained in \cite{MyID:345,fdf}.
\end{remark}

Most results about Lyapunov inequalities, including classical forms and
fractional continuous and discrete versions, assume, similarly to 
Theorem~\ref{th0}, the existence of a nontrivial solution to the considered problem.
In the following theorem, we give other assumptions, instead of assuming
existence of a nontrivial solution, to have new Lyapunov inequalities when
the nonlinear term satisfies certain conditions.

\begin{theorem}
\label{co} 
Consider the discrete fractional boundary value problem
\begin{equation*}
\left\{
\begin{array}{c}
\triangle ^{\alpha }y+q(t+\alpha -1)f(y(t+\alpha -1))=0,
\quad 1<\alpha \leq 2,\\
y(\alpha -2)=y(\alpha +b+1)=0,\quad \mathbb{N} \ni b \geq 2,
\end{array}
\right.
\end{equation*}
where $f\in C\left( \mathbb{R}_{+},\mathbb{R}_{+}\right)$ is nondecreasing
and $q:\mathbb{[\alpha -}1, \alpha +b]_{\mathbb{N}_{\alpha -1}}\rightarrow
\mathbb{R}_{+}$ is a nontrivial function. If there exist two positive
constants $r_{2}>r_{1}>0$ such that $f(y)\geq \overset{\ast }{\gamma }r_{1}$
for $y\in \lbrack 0,r_{1}]$ and $f(y)\leq \gamma r_{2}$ for $y\in \lbrack
0,r_{2}]$, then
\begin{equation}  
\label{eq:LI:d2}
\left\{
\begin{array}{l}
\sum\limits_{s=0}^{b+1}\left\vert q(s+\alpha -1)\right\vert 
> \frac{r_{1}}{\gamma r_{2}}\frac{4\Gamma (\alpha)
\Gamma (b+\alpha +2)\Gamma ^{2}(\frac{b}{2} +2)}{(b+2\alpha)(b+2)
\Gamma^{2}(\frac{b}{2}+\alpha )\Gamma (b+3)}\text{ if }
b\text{ is even},\\[0.3cm]
\sum\limits_{s=0}^{b+1}\left\vert q(s+\alpha -1)\right\vert 
> \frac{r_{1}}{\gamma r_{2}}\frac{\Gamma (\alpha )\Gamma (b+\alpha +2) 
\Gamma ^{2}(\frac{b+3}{2})}{\Gamma (b+3)\Gamma ^{2}(\frac{b+1}{2}+\alpha)} 
\text{ if }b\text{ is odd}.
\end{array}
\right.
\end{equation}
\end{theorem}

\begin{proof}
Follows from Theorems~\ref{th} and \ref{th0}.
\end{proof}

\begin{example}
Consider the following fractional boundary value problem:
\begin{equation*}
\left\{
\begin{array}{c}
\triangle ^{\frac{3}{2}}y(t)+\frac{2t+1}{2\Gamma (6)}\ln(2+y)=0,\\
y(-\frac{1}{2})=y(\frac{11}{2})=0.
\end{array}
\right.
\end{equation*}
We have that
\begin{description}
\item[$(i)$] $f(y)=\frac{\ln (2+y)}{\Gamma (6)}:\mathbb{R}_{+}\rightarrow
\mathbb{R}_{+}$ is continuous and nondecreasing;

\item[$(ii)$] $q(t)=\mathbb{[}\frac{1}{2},\frac{9}{2}]_{\mathbb{N} 
\frac{1}{2}}\rightarrow \mathbb{R}_{+}$ with $\sum\limits_{s=0}^{4} 
\frac{2s+1}{2}=\frac{25}{2}>0$.
\end{description}
We computed before the values of $\gamma $ and $\overset{\ast }{\gamma}$.
Choosing $r_{1}=1/10000$ and $r_{2}=1$, we get
\begin{enumerate}
\item $f(y)=\frac{\ln (2+y)}{\Gamma (6)}\geq \overset{\ast }{\gamma }r_{1}$
for $y\in \lbrack 0,1/10000]$;

\item $f(y)=\frac{\ln (2+y)}{\Gamma (6)}\leq \gamma r_{2}$ 
for $y\in \lbrack0,1]$.
\end{enumerate}
Therefore, from Theorem~\ref{co}, we get that
\begin{equation*}
\sum\limits_{s=0}^{4}\left\vert \frac{2s+1}{2}\right\vert 
> \frac{\Gamma(\frac{3}{2})\Gamma (\frac{13}{2})\Gamma^{2}(3) 
\Gamma (6)}{10000\Gamma(6)\Gamma ^{2}(\frac{7}{2})0.0616} 
\approx 0.15.
\end{equation*}
\end{example}


\section*{Acknowledgments}

This research was carried out while Chidouh was visiting the Department of
Mathematics of University of Aveiro, Portugal, 2016. The hospitality of the
host institution and the financial support of Houari Boumedienne University,
Algeria, are here gratefully acknowledged. Torres was supported through
CIDMA and the Portuguese Foundation for Science and Technology (FCT), within
project UID/MAT/04106/2013. The authors would like to thank an anonymous 
Referee for several comments and questions, which were useful to improve the paper. 



\end{document}